\documentclass[11pt,a4paper]{amsart}

\usepackage[margin=3cm,top=2.5cm,bottom=2cm]{geometry}

\usepackage{amsmath, amssymb, amsthm, mathrsfs}
\usepackage{graphicx,enumitem,microtype} 
\usepackage[dvipsnames]{xcolor}
\usepackage{mathabx}
\usepackage{bbm}
\usepackage[colorlinks=true,linkcolor=red!70!black,urlcolor=MidnightBlue,citecolor=MidnightBlue]{hyperref}

\usepackage[utf8]{inputenc} 
\usepackage[T1]{fontenc}
\usepackage{lmodern}

\newcommand{\sign}{\text{ sign}}

\newcommand{\beqn}{\begin{equation}}
\newcommand{\eeqn}{\end{equation}}
\newcommand{\bear}{\begin{eqnarray}}
\newcommand{\eear}{\end{eqnarray}}
\newcommand{\bean}{\begin{eqnarray*}}
\newcommand{\eean}{\end{eqnarray*}}
\newcommand{\bal}{\begin{aligned}}
\newcommand{\eal}{\end{aligned}}

\numberwithin{equation}{section}

\setlist[enumerate]{wide,labelindent=0cm,label=\textnormal{(\arabic*)},itemsep=5pt,topsep=4pt}

\theoremstyle{plain}
\newtheorem{theo}{Theorem}[section]
\newtheorem{prop}[theo]{Proposition}

\newtheorem{coro}[theo]{Corollary}
\theoremstyle{remark}

\theoremstyle{definition}

\title[Averaging method for polynomial dynamical systems]{The Krylov-Bogoliuvob-Mitropolsky averaging method for polynomial dynamical systems}

\author[F. Alvarez]{Frank Alvarez Borges}
\address[Frank Alvarez Borges]{Laboratoire Jacques-Louis Lions (LJLL),
  Sorbonne Universit\'e UPMC, 4 Place Jussieu, 75005 Paris 5, France}
\email{frank.alvarez$\_$borges@sorbonne-universite.fr}

\author[M. Rodriguez-Ricard]{Mariano Rodriguez-Ricard}
\address[Mariano Rodriguez-Ricard]{Facultad de Matématica y Ciencias de la Computaci\'on, Universidad de La Habana, Neptuno y San Lazaro, La Habana, Cuba}
\email{rricard@matcom.uh.cu}

\date{\today}



\begin{document}

\begin{abstract}
We describe the transformation of a polynomial planar dynamical system into a second order differential equation by means of a polynomial change of variables. We then, by means of the Krylov-Bogoliubov-Mitropolsky averaging method, identify sufficient conditions involving said change of variables so that a limit cycle exists.   

\end{abstract}

\maketitle

\tableofcontents

\section*{Introduction}
The second part of Hilbert's sixteenth problem \cite{hilbert2019mathematical} is concerned with the number and relative position of the limit cycles of polynomial planar dynamical system. Said problem, together with the Riemann hypothesis and the solution of the 7th-degree equation using algebraic functions of two parameters are the only ones still classified as unresolved, out of the 23 original problems. This alone, serves as indication of the complexity of said problem. However, as is often the case, not being a {\it solved } problem, does not imply that a considerable amount of advances has been made during the time since its statement.\\
There are many results and methods around the existence (or non-existence) of limit cycles, as well as their number and position. Some examples are the Poincaré-Bendixson theorem \cite{bendixson1901courbes} and the  Bendixson-Dulac theorem \cite{dumortier2006qualitative} which are of qualitative nature, and the Lindtsedt-Poincaré method \cite{mickens1996oscillations} which is of quantitative nature.\\
Averaging methods, such as the one developed by Krylov, Bogoliubov and Mitropolsky (KBM for short) \cite{mickens1996oscillations}, are widely used in the literature on the study of limit cycles for differential equations. Said methods allow for the analysis of existence of the periodic solutions, together with their amount, distribution, stability and approximate expression. Some examples of their application are \cite{buica2007periodic,buicua2004averaging,chow2012methods,marsden2012hopf,sanders2007averaging,verhulst2012nonlinear} where the issue of existence is addressed, and \cite{li2003hilbert,pi2009limit} where the number, distribution and stability is studied. At its core, the KBM method approximates solutions of perturbations of second order differential equations by computing as many Fourier coefficients of the solution as necessary in order to obtain the desired order of approximation. Since every first order differential system of two equations has second order equivalent equation, the KBM can be also be applied to said systems, after an appropriate change of variables.\\
In Section \ref{CoV} we show how to construct such a change of variables $\mathcal{H}(X)$, after showing that it always exists, independently of the degree of the non-linearity. The idea behind the proof of this result is that, by choosing $m$ large enough, the problem of finding $\mathcal{H}(X)$ reduces to solving a homogeneous linear system with more unknowns than equations, thus allowing one to find an infinity of non-trivial solutions.\\
Hence, if the inverse of $\mathcal{H}(X)$ admits a series expansion over a region $D_{\mathcal{H}}(\tau_{\alpha})$, the aforementioned second order differential equation reads
\begin{equation}
    \ddot{z}-\tau_{\alpha}\dot{z}+\delta_{\alpha} z=G(z,\dot z)=\mathcal{G}_2\cdot\Lambda_2(z,\dot z)+\mathcal{G}_3\cdot\Lambda_3(z,\dot z)+\|(z,\dot z)\|^4,\label{varsigma10}
\end{equation}
over $D_{\mathcal{H}}(\tau_{\alpha})$ where $\Lambda_k(x_1,x_2)=(x^k_1,x^{k-1}_1x_2,\ldots,x_1x_2^{k-1},x_2^k)$.
Once the system has been transformed into a second-order equation, the KBM averaging method provides a tool for determining both the existence and an asymptotic approximation of limit cycles. In section \ref{KBMM}, we apply it and we obtain the main result of our paper:
\begin{theo}
Suppose that, for any sufficiently small $\tau_{\alpha}$, there exists a family of changes of variables $\mathcal{H}(X,\tau_{\alpha})$ and $r_*(\tau_{\alpha})>0$, $\lim\limits_{\tau_{\alpha}\rightarrow 0}\frac{|\tau_{\alpha}|^{1/2}}{r_*(\tau_{\alpha})}=0$, such that
    \begin{itemize}
        \item[{\it i)}] $\dot{\left(\Pi_1\mathcal{H}(X)\right)}=\Pi_2\mathcal{H}(X)$, avec $|\Gamma|\neq 0$.
        \item[{\it ii)}] The value of
        $$
       p_3(0):=\lim\limits_{\tau_{\alpha}\rightarrow 0} p_3(\tau_{\alpha}):=\frac{1}{2\pi}\int_0^{2\pi}\mathcal{G}_3\cdot\Lambda_3(\cos\phi,-\sqrt{\delta_{\alpha}}\sin\phi)\sin\phi\ d\phi,
    $$
    is finite and non-zero. 
        \item[{\it iii)}] $\mathcal{H}^{-1}(Y,\tau_{\alpha})$ admits a power series of \ $Y$ over a region $D_{\mathcal{H}}(\tau_{\alpha})$, satisfying $B(0,r_*(\tau_{\alpha}))\subset D_{\mathcal{H}}(\tau_{\alpha})$. 
    \end{itemize}
    Then, for any sufficiently small $\tau_{\alpha}$ and $\sign(\tau_{\alpha})=\sign(p_3(0))$, a non-trivial periodic solution for the equation \eqref{varsigma10} exists and an approximation to it, of order $|\tau_{\alpha}|$ is given by
    $$
    \bar z(t)=\sqrt{|\tau_{\alpha}|}r_0\sin (\omega_0 t).
    $$
    where 
    $$
   r_0=\sqrt\frac{\delta_{\alpha}}{2 |p_3(0)|},\ \omega_0=1-\frac{\tau_{\alpha}}{2}\frac{p_3(0)}{q_3(0)},
    $$ 
    with
    $$q_3(0):=\lim\limits_{\tau_{\alpha}\rightarrow 0}q_3(\tau_{\alpha}):= \frac{1}{2\pi}\int_0^{2\pi}\mathcal{G}_3\cdot\Lambda_3(\cos\phi,-\sqrt{\delta_{\alpha}}\sin\phi)\cos\phi\ d\phi.$$
    Furthermore, the periodic solution corresponds to the appearance of, at least, a limit cycle for system \eqref{ODE1}. If such limit cycle is unique on a neighborhood of order $\sqrt{|\tau_{\alpha}|}$, then is a stable limit cycle if $\tau_{\alpha}>0$ (supercritical Hopf bifurcation) and and unstable limit cycle if $\tau_{\alpha}<0$ (subcritical Hopf bifurcation).
\end{theo} 

\section{Existence and approximation of the limit cycle}
Consider the system of differential equations
\begin{equation}
    \left\{\begin{matrix}
\dot x_1&=&f(x_1,x_2,\alpha)\\
\dot x_2&=&g(x_1,x_2,\alpha)
\end{matrix}\right.\label{ODE1}
\end{equation}
where $\alpha\in\mathbb{R}^k$ is a $k$-dimensional parameter and $f,g:\mathbb{R}^2\times \mathbb{R}^k\rightarrow\mathbb{R}$ are analytic functions satisfying $f(0,0,\alpha)=g(0,0,\alpha)=0$ for all $\alpha\in \mathbb{R}^k$. In other words, $(0,0)$ is a steady state of system \eqref{ODE1} for all values of $\alpha$. Adopting the same notations as in \cite{SMMR}, we set
$$
X:=\begin{pmatrix}
    x_1\\
    x_2
\end{pmatrix},\quad
J_{\alpha}:=\begin{pmatrix}
    \partial_{x_1}f(0,0,\alpha)&\partial_{x_2}f(0,0,\alpha)\\
    \partial_{x_1}g(0,0,\alpha)&\partial_{x_2}g(0,0,\alpha)
\end{pmatrix} \mbox{ and } \varPsi(X)=\begin{pmatrix}
    f(X,\alpha)\\
    g(X,\alpha)
\end{pmatrix}-J_{\alpha}X,
$$
so system \eqref{ODE1} takes the form
\begin{equation}
    \dot{X}=\mathcal{F}(X):=J_{\alpha}X+\varPsi(X).\label{ODEMform}
\end{equation}
Denote $\tau_{\alpha}:= tr(J_{\alpha})$ and $\delta_{\alpha}:=det(J_{\alpha})$. If there exists $\alpha_{0}$ such that $\tau_{\alpha_0}=0$, $\tau^2_{\alpha_0}-4\delta_{\alpha_0}<0$ and $\frac{\partial }{\partial_{\alpha_i}}\tau_{\alpha}\big|_{\alpha=\alpha_0}\neq 0$, for some $i\in\{1,\ldots,k\}$, then a Hopf bifurcation ( also known as Poincaré-Andronov-Hopf bifurcation) occurs around $\alpha_0$ and periodic orbits around the steady state $(0,0)$ appear for $\alpha$ on a vicinity of $\alpha_0$ (see \cite{hale2012dynamics}).\\
Averaging methods, such as the Krylov-Bogoliuvob-Mitropolski one (see \cite{mickens1996oscillations}) provide a tool to study the stability and asymptotic expansion in powers of $\tau_{\alpha}$ of said periodic orbits. To proceed, let us do in \eqref{ODEMform}, an invertible analytical transform of coordinates of a neighborhood of the origin onto another
\begin{equation}
    Y=\mathcal{H}(X)=\varGamma X+\mathcal{G}(X),\label{25}
\end{equation}
driven by a non-singular matrix $\varGamma$, and the analytic vector function:
\begin{equation}
   \mathcal{G}(X)=\begin{pmatrix}
    \overline{\varphi}(U,V)\\ \overline{\psi}(U,V)
   \end{pmatrix}=\begin{pmatrix}
    \sum\limits_{n=2}^{\infty}\sum\limits_{i+j=n}\overline{\varphi}_{ij}U^iV^j\\
    \sum\limits_{n=2}^{\infty}\sum\limits_{i+j=n}\overline{\psi}_{ij}U^iV^j
   \end{pmatrix} \label{26}
\end{equation}
which is assumed to have a positive radius of convergence. We shall denote the inverse
\begin{equation}
   X=\mathcal{H}^{-1}(Y)=\varGamma^{-1}Y+\mathcal{K}(Y), \label{27}
\end{equation}
where
\begin{equation}
    \mathcal{K}(Y)=\begin{pmatrix}
    \underline{\varphi}(z,\dot{z})\\ \underline{\psi}(z,\dot{z})
   \end{pmatrix}=\begin{pmatrix}
    \sum\limits_{n=2}^{\infty}\sum\limits_{i+j=n}\underline{\varphi}_{ij}z^i(\dot{z})^j\\
    \sum\limits_{n=2}^{\infty}\sum\limits_{i+j=n}\underline{\psi}_{ij}z^i(\dot{z})^j
   \end{pmatrix}\label{28}
\end{equation}
If $\mathcal{H}$ is such that
\begin{equation}
    Y=\begin{pmatrix}
     z\\\dot{z}
    \end{pmatrix}\label{29}
\end{equation}
for some function $z(t)$, the integration of the system can be reduced to the integration of a second order differential equation in the variable $z$.\\
The change of variables $\mathcal{H}$ verifies \eqref{29} if and only if
\begin{equation}
   \frac{d}{dt}(\varPi_1\mathcal{H})=\varPi_2\mathcal{H}.\label{33}
\end{equation}
Equation \eqref{33} implies that the components $\gamma_{ij}$ of $\varGamma$ verify the following concordance condition with the Jacobian of the system at the steady state
\begin{equation}
    J_{\alpha}^T\begin{pmatrix}
     \gamma_{11}\\ \gamma_{12}
    \end{pmatrix}=\begin{pmatrix}
     \gamma_{21}\\ \gamma_{22}
    \end{pmatrix}.\label{34}
\end{equation}
The concordance condition can be satisfied by choosing $\Gamma=a\Gamma_1+b\Gamma_2$, a linear combination of
\[
\Gamma_1=\begin{pmatrix}
1&0\\
j_{11}&j_{12}
\end{pmatrix}
\mbox{ and }
\Gamma_2=\begin{pmatrix}
0&1\\
j_{21}&j_{22}
\end{pmatrix},
\]
where $j_{kl}$ are the components of $J_{\alpha}$. Developing the left hand side term in \eqref{33} we get
\begin{align*}
    \frac{d}{dt}(\varPi_1\mathcal{H})&=\varPi_1\dot{\mathcal{H}}\\
    &=\varPi_1(\varGamma \dot{X}+\dot{\mathcal{G}})\\
    &=\varPi_1(\varGamma(J_{\alpha}X+\varPsi(X)))+\nabla\overline{\varphi}\cdot \dot{X}\\
    &=\varPi_1(\varGamma J_{\alpha}X)+\varPi_1(\varGamma\varPsi(X))+\nabla\overline{\varphi}\cdot \dot{X}.
\end{align*}
On the other hand, the right term on \eqref{33} is equal to
\begin{equation}
    \varPi_2\mathcal{H}=\varPi_2 \varGamma X+\overline{\psi}(X).
\end{equation}
Thanks to \eqref{34}, we know that $\varPi_1(\varGamma J_{\alpha}X)=\varPi_2 \varGamma X$, and, in conclusion we get the relation
\begin{equation}
    \overline{\psi}(X)=\varPi_1(\varGamma\varPsi(X))+\nabla\overline{\varphi}\cdot \dot{X}=\varPi_1(\varGamma\varPsi(X))+\frac{\partial\overline{\varphi}}{\partial U}\dot{U}+\frac{\partial\overline{\varphi}}{\partial V}\dot{V}.\label{PsiPhi}
\end{equation}
The equation satisfied by $z(t)$ is then
\begin{align}
    \ddot{z}=\dot{\Pi_2 \mathcal{H}(X)}=&\ \Pi_2 \left(\Gamma \dot{X}+\nabla\mathcal{G}(X)\cdot \dot{X}\right)\nonumber\\
    =&\ \Pi_2 \left(\Gamma J_{\alpha} X+\Gamma \varPsi (X)+\nabla\mathcal{G}(X)\cdot\mathcal{F}(X)\right)\nonumber\\
    =&\ \Pi_2 \left(\Gamma J_{\alpha} \Gamma^{-1}Y+\Gamma J_{\alpha}\mathcal{K}(Y)+\Gamma \varPsi (\mathcal{H}^{-1}(Y))+\nabla\mathcal{G}(\mathcal{H}^{-1}(Y))\cdot\mathcal{F}(\mathcal{H}^{-1}(Y))\right)\nonumber\\
    =&\ \tau_{\alpha}\dot{z}-\delta_{\alpha}z+ \Pi_2 \left(\Gamma J\mathcal{K}(Y)+\Gamma \varPsi (\mathcal{H}^{-1}(Y))+\nabla\mathcal{G}(\mathcal{H}^{-1}(Y))\cdot\mathcal{F}(\mathcal{H}^{-1}(Y))\right),\label{zdotdot}
\end{align}
for which is possible to obtain an approximation of the limit cycle by means of the Krylov–Bogoliubov-Mitropolski averaging method.\\
In \cite{SMMR} was remarked that if $\Gamma$ can be chosen in such a way that $\varPi_1(\varGamma\varPsi(X))=\mathcal{O}(\|X\|^M)$, with $M$ sufficiently high, then, by setting as $0$ the coefficients of $\overline{\varphi}$ associated to the powers of order lower than $M$, then, the coefficients of $\overline{\psi}(X)$ associated to those same powers, would also be $0$. As a consequence, the approximation and stability analysis of the periodic orbits only depend on the term $\Pi_2\left(\Gamma \varPsi (\Gamma^{-1}Y)\right)$, which only requires us to know the linear part $\Gamma X$ of the change of variables $\mathcal{H}(X)$. An example of a system where such a change of variables can be found is the Lengyel-Epstein reaction system, showcased in \cite{sarria2021bifurcations}. \\
However, is not always possible to guarantee the existence of a change of variables satisfying the aforementioned condition, and in those cases, the approximation and stability analysis of the periodic orbits depends on terms beyond the linear one. We aim on the following sections to develop a methodology, relying on a polynomial choice of $\mathcal{H}(X)$, which allows to address the asymptotic study of periodic orbits in the general case (i.e. $\varPi_1(\varGamma\varPsi(X))$ of any order).

\subsection{Polynomial non-linearity}\label{CoV}
 Given that any analytic non-linearity can be approximated by a polynomial expression of finite degree, as done for the tumor-immune system interaction system presented in \cite{CyC}, we restrict our study to that of polynomial non-linearities. Our first result states that for any polynomial non-linearity there exists a polynomial change of variables $\mathcal{H}$ which satisfies condition \eqref{33}.
\begin{prop}\label{PolyH}
    For any polynomial non-linearity $\varPsi(X)$ of degree $n$, there exists a non-trivial polynomial change of variables $\mathcal{H}$ of degree $m$, with
    \begin{eqnarray}
        m\leqslant \left\lceil\frac{2n-5+\sqrt{8n^2-16n+25}}{2}\right\rceil,\label{mmin}
    \end{eqnarray}
    such that condition \eqref{33} is satisfied.
\end{prop}
\begin{proof}
    The idea behind the proof of Proposition \ref{PolyH} is the fact that choosing $m$ sufficiently big, condition \eqref{33} is reduced to an homogeneous linear system with more unknowns than equations, hence making it possible to find infinitely many non trivial solutions. We give the detailed proof in what follows, in order to introduce some elements that will be of use later on.\\
    If $\varPsi(X)$ is a polynomial function, then there exist $n\in\mathbb{N}$ and matrices $\varphi_k\in \mathcal{M}_{2\times(k+1)}(\mathbb{R})$, $k=2,3,\ldots,n$, with $\varphi_n\neq \mathbf{0}_{2\times(k+1)}$ such that
    $$
    \varPsi(X)=\sum_{k=2}^{n}\varphi_k\Lambda_k(U,V),
    $$
    where $\Lambda_k(U,V):=(U^k,U^{k-1}V,\ldots,UV^{k-1},V^{k})^T$.\\
    The family of operators $\Lambda_k(U,V)$ has some useful properties. For example
    \begin{itemize}
        \item \begin{equation}
            \dot{\Lambda_k}(U,V)=\ \dot{U}\begin{pmatrix}
                kU^{k-1}\\
                (k-1)U^{k-2}V\\
                \vdots\\
                V^{k-1}\\
                0
            \end{pmatrix}+\dot{V}\begin{pmatrix}
                0\\
                U^{k-1}
                \\
                \vdots\\
                (k-1)UV^{k-2}\\
                kV^{k-1}
            \end{pmatrix}
            =\left(\dot{U}R_k+\dot{V}L_k\right)\Lambda_{k-1}(U,V),\label{dotLambda}
        \end{equation}
        where $R_k,L_k\in \mathcal{M}_{(k+1)\times k}(\mathbb{R})$ are defined as
        $$
        R_k:=\begin{pmatrix}
            k&0&\cdots&0&0\\
            0&k-1&\cdots&0&0\\
            \vdots&\vdots&\ddots&\vdots&\vdots\\
            0&0&\cdots&2&0\\
            0&0&\cdots&0&1\\
            0&0&\cdots&0&0
        \end{pmatrix} \mbox{ and }L_k:=\begin{pmatrix}
            0&0&\cdots&0&0\\
            1&0&\cdots&0&0\\
            0&2&\cdots&0&0\\
            \vdots&\vdots&\ddots&\vdots&\vdots\\
            0&0&\cdots&k-1&0\\
            0&0&\cdots&0&k
        \end{pmatrix}
        $$
        \item 
        \begin{equation}
            U^p\Lambda_k(U,V)=\begin{pmatrix}
                U^{k+p}\\
                U^{k+p-1}V\\
                \vdots\\
                U^{p-1}V^{k-1}\\
                U^pV^k
            \end{pmatrix}=\widehat{S}_{k,p}\Lambda_{k+p}(U,V)\label{UpLamk}
        \end{equation}
        where $$\widehat{S}_{k,p}\in\mathcal{M}_{(k+1)\times (k+p+1)}(\mathbb{R})$$ is defined as the block matrix $\widehat{S}_{k,p}:=\begin{pmatrix}\mathbf{I}_{k+1}&\mathbf{0}_{(k+1)\times p}\end{pmatrix}$.\\
        Similarly, we obtain
    \begin{equation}
            V^p\Lambda_k(U,V)=\widecheck{S}_{k,p}\Lambda_{k+p}(U,V)\label{VpLamk}
        \end{equation}
        where $\widecheck{S}_{k,p}:=\begin{pmatrix}\mathbf{0}_{(k+1)\times p}&\mathbf{I}_{k+1}\end{pmatrix}$.\label{DownLamk}
    \end{itemize}
Let us consider a change of variables of the form
$$
\mathcal{H}(X)=\varGamma X+\sum_{k=2}^{m}\Theta_k\Lambda_k(U,V),
$$
where $\Gamma$ satisfies the concordance condition, and the values of $m\in\mathbb{N}$ and $\Theta_k\in\mathcal{M}_{2\times(k+1)}(\mathbb{R})$ will be fixed later in such a way that condition \eqref{33} is fulfilled.\\
Adopting the notation $\Theta_1:=\Gamma$, we notice that
$$
\varPi_1\mathcal{H}=\sum_{k=1}^{m}\langle e_1,\Theta_k\Lambda_k(U,V)\rangle,
$$
where $\langle\cdot,\cdot\rangle$ represents the euclidean scalar product in $\mathbb{R}^2$ and $e_1$ is the first element of the canonical basis on the same space. Deriving with respect to $t$, we get
\begin{equation}
\dot{\varPi_1\mathcal{H}}=\sum_{k=1}^{m}\langle e_1,\Theta_k\dot{\Lambda_k}(U,V)\rangle=\langle e_1,\Gamma\dot{\Lambda_1}(U,V)\rangle+\sum_{k=2}^{m}\langle e_1,\Theta_k\dot{\Lambda_k}(U,V)\rangle \label{dotH1}    
\end{equation}
On the first place, thanks to the concordance condition, we have
\begin{align*}
    \langle e_1,\Gamma\dot{\Lambda_1}(U,V)\rangle=\langle \Gamma^Te_1,\dot{\Lambda_1}(U,V)\rangle=&\ \langle \Gamma^Te_1,J_{\alpha}X+\sum_{k=2}^{n}\varphi_k\Lambda_k(U,V)\rangle\\
    =&\ \langle J^T_{\alpha}\Gamma^Te_1,X\rangle+\sum_{k=2}^{n}\langle \Gamma^Te_1,\varphi_k\Lambda_k(U,V)\rangle\\
    =&\  \langle \Gamma^Te_2,X\rangle+\sum_{k=2}^{n}\langle \varphi^T_k\Theta^T_1e_1,\Lambda_k(U,V)\rangle.
\end{align*}
On the other hand, thanks to properties \eqref{dotLambda}, \eqref{UpLamk} and \eqref{DownLamk} we have
\begin{align*}
\sum_{k=2}^{m}\langle e_1,\Theta_k\dot{\Lambda_k}(U,V)\rangle=&\ \sum_{k=2}^{m}\langle \Theta^T_ke_1,\dot{\Lambda_k}(U,V)\rangle\\
=&\ \sum_{k=2}^{m}\langle \Theta^T_ke_1,\left(\dot{U}R_k+\dot{V}L_k\right)\Lambda_{k-1}(U,V)\rangle\\
=&\ \sum_{k=2}^{m}\langle \left(\dot{U}R^T_k+\dot{V}L^T_k\right)\Theta^T_ke_1,\Lambda_{k-1}(U,V)\rangle\\
=&\ \dot{U}\sum_{k=2}^{m}\langle R^T_k\Theta^T_ke_1,\Lambda_{k-1}(U,V)\rangle+\dot{V}\sum_{k=2}^{m}\langle L^T_k\Theta^T_ke_1,\Lambda_{k-1}(U,V)\rangle.
\end{align*}
Denoting $\varphi_1:=J_{\alpha}$, then
$$
\dot{U}=\sum\limits_{j=1}^n\langle e_1,\varphi_j\Lambda_j(U,V)\rangle=\sum\limits_{j=1}^n\sum\limits_{l=1}^{j+1}\varphi_j(1,l)U^{j-l+1}V^{l-1}
$$
which, thanks to \eqref{UpLamk} and \eqref{DownLamk}, gives
\begin{align*}
    &\dot{U}\sum_{k=2}^{m}\langle R^T_k\Theta^T_ke_1,\Lambda_{k-1}(U,V)\rangle\\
    =&\ \left(\sum\limits_{j=1}^n\sum\limits_{l=1}^{j+1}\varphi_j(1,l)U^{j-l+1}V^{l-1}\right)\left(\sum_{k=2}^{m}\langle R^T_k\Theta^T_ke_1,\Lambda_{k-1}(U,V)\rangle\right)\\
    =&\ \sum_{k=2}^{m}\Big\langle R^T_k\Theta^T_ke_1,\sum\limits_{j=1}^n\sum\limits_{l=1}^{j+1}\varphi_j(1,l)U^{j-l+1}V^{l-1}\Lambda_{k-1}(U,V)\Big\rangle\\
    =&\ \sum_{k=2}^{m}\Big\langle R^T_k\Theta^T_ke_1,\sum\limits_{j=1}^n\sum\limits_{l=1}^{j+1}\varphi_j(1,l)U^{j-l+1}\widecheck{S}_{k-1,l-1}\Lambda_{k+l-2}(U,V)\Big\rangle\\
    =&\ \sum_{k=2}^{m}\Big\langle R^T_k\Theta^T_ke_1,\sum\limits_{j=1}^n\sum\limits_{l=1}^{j+1}\varphi_j(1,l)\widecheck{S}_{k-1,l-1}\widehat{S}_{k+l-2,j-l+1}\Lambda_{k+j-1}(U,V)\Big\rangle\\
    =&\ \sum_{k=2}^{m}\Big\langle R^T_k\Theta^T_ke_1,\sum\limits_{j=1}^nT_{1,j,k}\Lambda_{k+j-1}(U,V)\Big\rangle,
\end{align*}
where
$$
T_{1,j,k}:=\sum\limits_{l=1}^{j+1}\varphi_j(1,l)\widecheck{S}_{k-1,l-1}\widehat{S}_{k+l-2,j-l+1}\in\mathcal{M}_{k\times(j+k)}(\mathbb{R}).
$$
For $k=2,\ldots,m$ and $j=1,\ldots,n$, the previous relation reads
\begin{align*}
    \dot{U}\sum_{k=2}^{m}\langle R^T_k\Theta^T_ke_1,\Lambda_{k-1}(U,V)\rangle=&\ \sum_{k=2}^{m}\Big\langle R^T_k\Theta^T_ke_1,\sum\limits_{j=1}^n T_{1,j,k}\Lambda_{k+j-1}(U,V)\Big\rangle\\
    =&\ \sum_{k=2}^{m}\sum\limits_{j=1}^n\Big\langle R^T_k\Theta^T_ke_1, T_{1,j,k}\Lambda_{k+j-1}(U,V)\Big\rangle\\
    =&\ \sum_{k=2}^{m}\sum\limits_{j=1}^n\Big\langle T^T_{1,j,k}R^T_k\Theta^T_ke_1, \Lambda_{k+j-1}(U,V)\Big\rangle.
\end{align*}
The double sum can be re-arranged so
\begin{equation}
    \dot{U}\sum_{k=2}^{m}\langle R^T_k\Theta^T_ke_1,\Lambda_{k-1}(U,V)\rangle=\sum_{k=2}^{m+n-1}\Big\langle \sum\limits_{j=1,\ldots,n,\atop l=2,\ldots,m}^{j+l-1=k}T^T_{1,j,l}R^T_l\Theta^T_le_1, \Lambda_{k}(U,V)\Big\rangle\label{TermA}
\end{equation}
A similar computation gives
\begin{equation}
    \dot{V}\sum_{k=2}^{m}\langle L^T_k\Theta^T_ke_1,\Lambda_{k-1}(U,V)\rangle=\sum_{k=2}^{m+n-1}\Big\langle \sum\limits_{j=1,\ldots,n,\atop l=2,\ldots,m}^{j+l-1=k}T^T_{2,j,l}L^T_l\Theta^T_le_1, \Lambda_{k}(U,V)\Big\rangle\label{TermB}
\end{equation}
where
$$
T_{2,j,k}:=\sum\limits_{l=1}^{j+1}\varphi_j(2,l)\widecheck{S}_{k-1,l-1}\widehat{S}_{k+l-2,j-l+1},
$$
for $k=2,\ldots,m$ and $j=1,\ldots,n$.\\
Putting \eqref{TermA} and \eqref{TermB} together, we get
$$
\sum_{k=2}^{m}\langle e_1,\Theta_k\dot{\Lambda_k}(U,V)\rangle=\sum_{k=2}^{m+n-1}\Big\langle \sum\limits_{j=1,\ldots,n,\atop l=2,\ldots,m}^{j+l-1=k}\left(T^T_{1,j,l}R^T_l+T^T_{2,j,l}L^T_l\right)\Theta^T_le_1, \Lambda_{k}(U,V)\Big\rangle.
$$
In conclusion
\begin{align*}
    \dot{\varPi_1\mathcal{H}}=&\langle \Gamma^Te_2,X\rangle+\sum_{k=2}^{n}\langle \varphi^T_k\Theta^T_1e_1,\Lambda_k(U,V)\rangle\\
    &+\sum_{k=2}^{m+n-1}\Big\langle \sum\limits_{j=1,\ldots,n,\atop l=2,\ldots,m}^{j+l-1=k}\left(T^T_{1,j,l}R^T_l+T^T_{2,j,l}L^T_l\right)\Theta^T_le_1, \Lambda_{k}(U,V)\Big\rangle.\label{fulldotH1}
\end{align*}
Given that
$$
\varPi_2\mathcal{H}=\langle \Gamma^Te_2,X\rangle+\sum_{k=2}^m\langle e_2,\Theta_k\Lambda_k(U,V)\rangle=\langle \Gamma^Te_2,X\rangle+\sum_{k=2}^m\langle \Theta^T_ke_2,\Lambda_k(U,V)\rangle,
$$
condition \eqref{33} can be satisfied if the components of $\Theta_k$ are chosen as the solution of the linear system
\begin{equation}
\left(\varphi^T_k\Theta^T_1e_1\right)\mathbbm{1}_{k\leqslant n}+\sum\limits_{j=1,\ldots,n,\atop l=2,\ldots,m}^{j+l-1=k}\left(T^T_{1,j,l}R^T_l+T^T_{2,j,l}L^T_l\right)\Theta^T_le_1-\left(\Theta^T_k e_2\right)\mathbbm{1}_{k\leqslant m}=0,
\label{mpn}
\end{equation}
for $2\leqslant k\leqslant m+n-1$. This is an homogeneous linear system with $m^2+3m-2$ unknowns and at most $\frac{m^2+(2n+1)m+n(n+1)-6}{2}$ equations. This means that for $m$ big enough, the system is under-determined and hence, it will have infinitely many non-trivial solutions. In fact, it suffices that
$$
m^2+3m-2=\frac{m^2+(2n+1)m+n(n+1)-6}{2}+1
$$
in order to guarantee the existence of said non trivial solutions (one more unknown than equations). Solving for $m$ we get the bound \eqref{mmin}.
\end{proof} 
The degree of the change of variables does not have necessarily to be equal to the bound given in Proposition \ref{PolyH}. The next corollary gives a sufficient condition for the existence of non trivial polynomial change of variables of arbitrary degree $m$.\\
\begin{coro}\label{corommin}
    For a non-linearity $\varPsi(X)$ of degree $n$, there exists a non-trivial polynomial change of variables $\mathcal{H}$ of degree $m$ such that condition \eqref{33} is satisfied, if the rank of the matrix associated to system \eqref{mpn} is strictly smaller than $m^2+3m-2$.
\end{coro}
In general, a polynomial change of variables does not have to be invertible in a vicinity of the origin, however, it is well known that, if the linear term has non-zero determinant, then an inverse exists on a neighborhood of $(0,0)$. Furthermore, on a subset of said neighborhood, the inverse admits a power series representation. As a final result for this section, we explicitly give, as functions of $\Gamma=\Theta_1$, $\Theta_2$ and $\Theta_3$, the first terms for the power series of $\mathcal{H}^{-1}(Y)$, when it exists.
\begin{theo}\label{psinvH}
    If there exists a non-trivial polynomial change of variables $\mathcal{H}$ of degree $m$ such that condition \eqref{33} is satisfied, and $|\Gamma|\neq 0$, then $\mathcal{H}$ is invertible in a neighborhood $\mathcal{N}_{\mathcal{H}}$ of $(0,0)$. Furthermore, there exists a subset $D_{\mathcal{H}}\subset \mathcal{N}_{\mathcal{H}}$ such that $\mathcal{H}^{-1}(Y)$ admits an expansion in powers of $Y$. The first three terms of said expansion are
    $$
    \mathcal{H}^{-1}(Y)=\Gamma^{-1}Y+\Xi_2\Lambda_2(Y)+\Xi_3\Lambda_3(Y)+\mathcal{O}(\|Y\|^4),
    $$
    where
    $$
    \Xi_2=-\Gamma^{-1}\Theta_2 \mathcal{P}_2\left(\Gamma^{-1}\right)\mbox{ and }\Xi_3=-\Gamma^{-1}\left(\Theta_2\mathcal{R}_2(\Gamma^{-1},\Xi_2)+\Theta_3\mathcal{P}_3\left(\Gamma^{-1}\right)\right),
    $$
    and
    \begin{align*}
        \mathcal{P}_1\left(A\right)&=A,\\\mathcal{P}_k\left(A\right)&=\frac{1}{k}\left(R_k\mathcal{P}_{k-1}(A)\left(\widehat{S}_{k-1,1}A_{11}+\widecheck{S}_{k-1,1}A_{12}\right)\right)\\
        &+\frac{1}{k}\left(L_k\mathcal{P}_{k-1}(A)\left(\widehat{S}_{k-1,1}A_{21}+\widecheck{S}_{k-1,1}A_{22}\right)\right),\ k>1,\\
        \mathcal{R}_2(A,B)&=\left(R_2 B\left(\widehat{S}_{2,1}A_{11}+\widecheck{S}_{2,1}A_{12}\right)+L_2 B\left(\widehat{S}_{2,1}A_{21}+\widecheck{S}_{2,1}A_{22}\right)\right).
    \end{align*}   
\end{theo}
\begin{proof}
    The result is a direct consequence of the inverse function theorem. To obtain the value of the coefficients it suffices to replace the power series of both $\mathcal{H}$ and $\mathcal{H}^{-1}$ on the relation $Y=\mathcal{H}(\mathcal{H}^{-1}(Y))$ and equate the coefficients of similar powers of $Y$.
\end{proof}
\section{The Krylov-Bogoliubov-Mitropolski averaging method}\label{KBMM}
The Krylov-Bogoliubov-Mitropolski averaging method is a mathematical method for approximate analysis of oscillating processes in non-linear mechanics. It generalizes the averaging Krylov-Bogoliubov method in order to obtain approximations of any desired order of $\varepsilon$ for the differential equation
\begin{equation}
    \ddot{u}+u=\sum_{i=1}^N\varepsilon^i f_i(u,\dot{u}),\label{equeps}
\end{equation}
where the values of $N$ and $f_i$ are known.\\
We briefly showcase the method yielding an approximation of order $\varepsilon$. This order of approximation will allow to derive an approximation of order $\sqrt{\tau_{\alpha}}$ for the periodic orbit of the reaction system \eqref{ODE1}.\\
Following the procedure in \cite{mickens1996oscillations}, in the Krylov-Bogoliubov-Mitropolski method, the solution is assumed to have the form
\begin{equation}
   u(t)=r(t)\sin (\phi(t)),\label{ueps} 
\end{equation}
where the quantities $r(t)$ and $\phi(t)$ are functions of time defined by the following equations:
\begin{align}
\frac{d}{dt}r(t)&=\varepsilon A_1(r(t)),\label{eqrKBM}\\
\frac{d}{dt}\phi(t)&=1+\varepsilon B_1(r(t)).\label{eqphiKBM}
\end{align}
The functions $A_1(a)$ and $B_1(a)$, $i=1,2$ are to be chosen in a way that, after replacing \eqref{ueps} in \eqref{equeps}, this last equation is satisfied up to the terms of order $\varepsilon$.\\
After doing so, as shown in \cite{mickens1996oscillations}, we obtain
\begin{align*}
    A_1(r)&=-\left(\frac{1}{2\pi}\right)\int_0^{2\pi}f_1(r\cos\phi,-r\sin\phi)\sin\phi\ d\phi,\\
    B_1(r)&=-\left(\frac{1}{2\pi r}\right)\int_0^{2\pi}f_1(r\cos\phi,-r\sin\phi)\cos\phi\ d\phi,
\end{align*}
If the equation for $r(t)$ has a positive steady state, then the associated value of $u(t)$ corresponds to an approximation of order $\varepsilon$ of a periodic orbit for equation \eqref{equeps}. Furthermore, the stability of the periodic orbit will be the same as the stability of the steady state for equation \eqref{eqrKBM}.\\
For this reason, the first step before applying the Krylov-Bogoliubov-Mitropolski averaging method, will be to derive a power series for the non linear term of \eqref{zdotdot}.\\
As in \eqref{29}, we set 
\begin{equation*}
    Y=\begin{pmatrix}
     z\\\dot{z}
    \end{pmatrix},
\end{equation*}
and we denote $\tau_{\alpha}:=tr(J_{\alpha})$, $\delta_{\alpha}:=det(J_{\alpha})$ and
$$
G(z,\dot{z})=G(Y):=\Pi_2 \left(\Gamma J_{\alpha}\mathcal{K}(Y)+\Gamma \varPsi (\mathcal{H}^{-1}(Y))+\nabla\mathcal{G}(\mathcal{H}^{-1}(Y))\cdot\mathcal{F}(\mathcal{H}^{-1}(Y))\right),
$$
so $z(t)$ is the solution of
$$
\ddot{z}-\tau_{\alpha}\dot{z}+\delta_{\alpha} z=G(z,\dot z).
$$
We are looking for oscillations with small amplitude $(\tau_{\alpha})^{\kappa}$, $\kappa>0$ to be fixed later, so after making the change of variables
 \[
 z(t)=(\tau_{\alpha})^{\kappa} \varsigma(\sqrt{\delta_{\alpha}}t),
 \]
 we get that $\varsigma$ satisfies the equation
 \begin{equation}
     \ddot{\varsigma}-\frac{\tau_{\alpha}}{\sqrt{\delta_{\alpha}}}\dot{\varsigma}+\varsigma= \frac{1}{(\tau_{\alpha})^{\kappa}\delta_{\alpha}} G((\tau_{\alpha})^{\kappa}\varsigma,(\tau_{\alpha})^{\kappa}\sqrt{\delta_{\alpha}}\dot{\varsigma}).\label{varsigma1}
 \end{equation}
 Thanks to Proposition \ref{psinvH}, the power series for $G(Y)$ is given by the expression
 $$
 G(Y)=\mathcal{G}_2\cdot\Lambda_2(Y)+\mathcal{G}_3\cdot\Lambda_3(Y)+\|Y\|^4
 $$
 where
\begin{align}
    \mathcal{G}_2=&\Pi_2\left(\Gamma J_{\alpha}\Xi_2+\Gamma\varphi_2\mathcal{P}_2(\Gamma^{-1})\right)+\mathcal{P}^T_2(\Gamma^{-1})\mathcal{S}_2(J_{\alpha},\Theta_2)\nonumber\\
    \mathcal{G}_3=&\Pi_2\left(\Gamma J_{\alpha}\Xi_3+\Gamma\varphi_2\mathcal{R}_2(\Gamma^{-1},\Xi_2)+\Gamma\varphi_3\mathcal{P}_3(\Gamma^{-1})\right)\nonumber\\
    &+\mathcal{R}^T_2(\Gamma^{-1},\Xi_2)\mathcal{S}_2(J_{\alpha},\Theta_2)+\mathcal{P}^T_3(\Gamma^{-1})\left(\mathcal{S}_3(J_{\alpha},\Theta_3)+\mathcal{T}(\Theta_2,\varphi_2)\right),\label{G3}
\end{align}
and
\begin{align*}
\mathcal{S}_k(A,B)=&\left(A_{11} \widehat{S}^T_{k-1,1}+A_{12}\widecheck{S}^T_{k-1,1}\right)R^T_k\Pi_2B\\
    &+\left(A_{21} \widehat{S}^T_{k-1,1}+A_{22}\widecheck{S}^T_{k-1,1}\right)L^T_k\Pi_2B,\ k=2,3,\\
    \mathcal{T}(A,B)=&\left(2A_{21} \widehat{S}^T_{2,1}+A_{22}\widecheck{S}^T_{2,1}\right)\Pi_1B\\
    &+\left(A_{22} \widehat{S}^T_{2,1}+2A_{23}\widecheck{S}^T_{2,1}\right)\Pi_2B.
\end{align*}
Therefore, setting $\kappa=1/2$ and $\varepsilon=\sqrt{|\tau_{\alpha}|}$, equation \eqref{varsigma1} takes the form
\begin{equation}
    \ddot{\varsigma}+\varsigma=\frac{\varepsilon}{\delta_{\alpha}}\mathcal{G}_2\cdot\Lambda_2(\varsigma,\sqrt{\delta_{\alpha}}\dot{\varsigma})+\varepsilon^2\left(\frac{\sign(\tau_{\alpha})}{\sqrt{\delta_{\alpha}}}\dot{\varsigma}+\frac{1}{\delta_{\alpha}}\mathcal{G}_3\cdot\Lambda_3(\varsigma,\sqrt{\delta_{\alpha}}\dot{\varsigma})\right)+\varepsilon^3\|(\varsigma,\sqrt{\delta_{\alpha}}\dot{\varsigma})\|^4,\label{varsigma2}
\end{equation}
for which, under certain conditions, we are able to construct the Krylov-Bogoliubov-Mitropolski approximation.\\
The following proposition states those conditions and explicitly give the value of an approximation of the averaged limit cycle for equation \ref{varsigma2}.

\begin{prop}\label{3rdLC}
Assume that, for all $\tau_{\alpha}$ sufficiently small, there exists a family of changes of variables $\mathcal{H}(X,\tau_{\alpha})$ and $r_*(\tau_{\alpha})>0$, $\lim\limits_{\tau_{\alpha}\rightarrow 0}\frac{|\tau_{\alpha}|^{1/2}}{r_*(\tau_{\alpha})}=0$, such that 
    \begin{itemize}
        \item[{\it i)}] $\mathcal{H}(X,\tau_{\alpha})$ satisfies condition \eqref{33}, with $|\Gamma|\neq 0$.
        \item[{\it ii)}] The value 
        $$
       p_3(0):=\lim\limits_{\tau_{\alpha}\rightarrow 0} p_3(\tau_{\alpha}):=\frac{1}{2\pi}\int_0^{2\pi}\mathcal{G}_3\cdot\Lambda_3(\cos\phi,-\sqrt{\delta_{\alpha}}\sin\phi)\sin\phi\ d\phi,
    $$
    is finite and non zero. 
        \item[{\it iii)}] $\mathcal{H}^{-1}(Y,\tau_{\alpha})$ admits an expansion in powers of \ $Y$ over a region $D_{\mathcal{H}}(\tau_{\alpha})$, satisfying $B(0,r_*(\tau_{\alpha}))\subset D_{\mathcal{H}}(\tau_{\alpha})$. 
    \end{itemize}
    Then, for all $\tau_{\alpha}$ sufficiently small such that $\sign(\tau_{\alpha})=\sign(p_3(0))$, a non-trivial periodic solution for equation \eqref{varsigma1} exists and an approximation for it, of order $|\tau_{\alpha}|^{1/2}$ is given by
    $$
    \bar\varsigma(t)=r_0\sin (\omega_0 t).
    $$
    where 
    $$
   r_0=\sqrt\frac{\delta_{\alpha}}{2 |p_3(0)|},\ \omega_0=1-\frac{\tau_{\alpha}}{2}\frac{p_3(0)}{q_3(0)},
    $$ 
    with
    $$q_3(0):=\lim\limits_{\tau_{\alpha}\rightarrow 0}q_3(\tau_{\alpha}):= \frac{1}{2\pi}\int_0^{2\pi}\mathcal{G}_3\cdot\Lambda_3(\cos\phi,-\sqrt{\delta_{\alpha}}\sin\phi)\cos\phi\ d\phi.$$
    Furthermore, the periodic solution corresponds to the appearance of, at least, a limit cycle for system \eqref{ODE1}. If such limit cycle is unique on a neighborhood of order $\sqrt{|\tau_{\alpha}|}$, then is a stable limit cycle if $\tau_{\alpha}>0$ (supercritical Hopf bifurcation) and and unstable limit cycle if $\tau_{\alpha}<0$ (subcritical Hopf bifurcation).
\end{prop}
\begin{proof}
From the right-hand side of \ref{varsigma2} is immediate that 
$$
A_1(r)= \sqrt{|\tau_{\alpha}|}\Big(\sign(\tau_{\alpha})\frac{r}{2}-\frac{p_3(\tau_{\alpha})}{\delta_{\alpha}}r^3\Big) \mbox{ and } B_1(r)=-\frac{\sqrt{|\tau_{\alpha}|}}{\delta_{\alpha}}q_3(\tau_{\alpha})r^2.
$$
Equation \eqref{eqrKBM} will have a positive steady state $$r_{\tau_{\alpha}}=\sqrt{\frac{ \delta_{\alpha} }{2\sign(\tau_{\alpha})p_3(\tau_{\alpha})}}$$ if $\sign(\tau_{\alpha})=\sign(p_3(\tau_{\alpha}))$.  Thanks to condition {\it ii)}, This value of $r_{\tau_{\alpha}}$ corresponds to the radius of periodic orbit which is not an spurious limit cycle (see \cite{mickens1996oscillations}). Furthermore, $r_{\tau_{\alpha}}$ induces the value
$$
\phi(t)=(1-\frac{\tau_{\alpha}}{2}\frac{q_3(\tau_{\alpha})}{p_3(\tau_{\alpha})})t=:\omega_{\tau_{\alpha}} t.
$$
We see that, for small values of $\tau_{\alpha}$, $r_{\tau_{\alpha}}=r_0+\mathcal{O}(|\tau_{\alpha}|)$ and $\omega_{\tau_{\alpha}}=\omega_0+\mathcal{O}(|\tau_{\alpha}|)$, hence, the order of approximation of the averaged solution does not change if we drop the terms of order $|\tau_{\alpha}|$. Finally, thanks to condition {\it iii)},  $\sqrt{|\tau_{\alpha}|}\bar\varsigma(t)\subset D_{\mathcal{H}}(\tau_{\alpha})$ for small values of $\tau_{\alpha}$, so $\mathcal{H}^{-1}(\sqrt{|\tau_{\alpha}|}\bar\varsigma(t),\sqrt{|\tau_{\alpha}|}\bar\varsigma'(t),\tau_{\alpha})$ is well defined and corresponds to the appearance of limit cycle for system \eqref{ODE1}, whose stability is the same as that of $r(t)$ in the averaged equation.
\end{proof}
An immediate consequence of Proposition \ref{3rdLC} is the fact that, if condition \eqref{33} can be satisfied with an invertible linear change of variables, then condition {\it iii)} is automatically satisfied, since the inverse would exist and be linear over $\mathbb{R}^2$. The case for linear (and quasi-linear) changes of variables was already studied in \cite{SMMR}. Our main contribution is the extension of the methodology to the general case of any polynomial change of variables.\\
Since the terms of order higher than $|\tau_{\alpha}|^{1/2}$ will not improve the approximation, we may conclude
\begin{prop}
    Under the hypotheses of Proposition \ref{3rdLC}, a limit cycle for system \ref{ODE1} exists and an approximation of order $\tau_{\alpha}$ for it is given by
    $$
    \bar X(t)=\Gamma^{-1}\Lambda_1(\sqrt{\tau_{\alpha}}r_0\sin(\sqrt{\delta_{\alpha}}\omega_0t),\sqrt{\tau_{\alpha}\delta_{\alpha}}r_0\omega_0\cos(\sqrt{\delta_{\alpha}}\omega_0t)).$$
\end{prop}
\newpage
\bibliographystyle{plain} 
\bibliography{references}
\end{document}